\documentclass[12pt,a4paper]{amsart}
\usepackage[latin9]{inputenc}
\usepackage{color}
\usepackage{amstext}
\usepackage{amsthm}
\usepackage{amssymb}
\usepackage{stmaryrd}
\usepackage{esint}
\usepackage[unicode=true,pdfusetitle,
 bookmarks=true,bookmarksnumbered=false,bookmarksopen=false,
 breaklinks=false,pdfborder={0 0 1},backref=false,colorlinks=false]
 {hyperref}
 
\makeatletter
\numberwithin{equation}{section}
\numberwithin{figure}{section}
\theoremstyle{plain}
\newtheorem{thm}{\protect\theoremname}
  \theoremstyle{plain}
  \newtheorem{prop}[thm]{\protect\propositionname}
  \theoremstyle{definition}
  \newtheorem{example}[thm]{\protect\examplename}
  \theoremstyle{plain}
  \newtheorem{lem}[thm]{\protect\lemmaname}
  \theoremstyle{definition}
  \newtheorem{defn}[thm]{\protect\definitionname}
  \newtheorem{rem}[thm]{\protect\remarkname}

\usepackage{amssymb,amsthm,amsmath,amsfonts,amscd}
\usepackage{graphicx}
\usepackage{url}
\usepackage{color}
\usepackage[active]{srcltx}
\usepackage[matrix,arrow]{xy}
\usepackage{mathrsfs}
\usepackage{enumerate}
\usepackage{amsopn} 
\usepackage{bbm} 
\usepackage[normalem]{ulem} 
\usepackage{cite}
\usepackage{hyperref}
\allowdisplaybreaks[1]
\usepackage[english]{babel}
\usepackage{bbm}
\usepackage{mhchem}
\usepackage{dsfont}

\usepackage{dsfont}\usepackage{xcolor}
\usepackage{csquotes}
\newcommand{\norm}[1]{{\left\Vert #1\right\Vert}}
\newcommand{\abs}[1]{{\left\vert #1\right\vert}}
\newcommand{\e}{\varepsilon}
\newcommand{\R}{{\mathbb R}}
\renewcommand{\C}{{\mathbb C}}
\newcommand{\one}{{\mathds{1}}}

\DeclareMathOperator{\supp}{supp}
\DeclareMathOperator{\dist}{dist}
\DeclareMathOperator{\sgn}{sgn}
\allowdisplaybreaks

\date{\today}

\makeatother

  \providecommand{\definitionname}{Definition}
  \providecommand{\examplename}{Example}
  \providecommand{\lemmaname}{Lemma}
  \providecommand{\propositionname}{Proposition}
\providecommand{\theoremname}{Theorem}
\providecommand{\remarkname}{Remark}

\begin{document}

\title[scalar conservation laws with a force]{Regularity of solutions to scalar conservation laws with a force}
\begin{abstract}
We prove regularity estimates for entropy solutions to scalar conservation laws with a force. Based on the kinetic form of a scalar conservation law, a new decomposition of entropy solutions is introduced, by means of a decomposition in the velocity variable, adapted to the non-degeneracy properties of the flux function. This allows a finer control of the degeneracy behavior of the flux. In addition, this decomposition allows to make use of the fact that the entropy dissipation measure has locally finite singular moments. Based on these observations, improved regularity estimates for entropy solutions to (forced) scalar conservation laws are obtained.
\end{abstract}

\author{Benjamin Gess}

\address{Max-Planck Institute for Mathematics in the Sciences\\
Inselstraße 22\\
04103 Leipzig\\
Germany}

\email{bgess@mis.mpg.de}

\author{Xavier Lamy}

\address{Institut de Mathématiques de Toulouse\\
Université Paul Sabatier\\
Toulouse\\
France}

\email{xlamy@math.univ-toulouse.fr}


\maketitle

\section{Introduction}

\label{s:intro}

We consider the regularity of solutions to scalar conservation laws
\begin{align}
\partial_{t}u+\mathrm{div}A(u) & =S\quad\text{on }(0,T)\times\R^{n}\label{eq:intro_SCL}\\
u(0) & =u_{0}\quad\text{on }\R^{n},\nonumber 
\end{align}
for $S\in L^{1}([0,T]\times\R^{n})$, $u_{0}\in L^{1}(\R^{n})$ and $A\in C^{2}(\R;\R^{n})$ satisfying a non-degeneracy condition to be specified below. 

In the special case, $n=1$, $S\equiv0$ and $A$ convex, the one-sided Oleinik inequality for entropy solutions can be used to obtain optimal regularity estimates for \eqref{eq:intro_SCL}. More precisely, assuming in addition that 
\[
\inf_{(u,v)\in\R^{2},\,u\ne v}\frac{|A'(u)-A'(v)|}{|u-v|^{\ell}}>0
\]
for some $l>0$, Bourdarias, Gisclon and Junca have shown in \cite{BGJ14} that bounded entropy solutions for \eqref{eq:intro_SCL} satisfy $u(t)\in W_{loc}^{\frac{1}{\ell}-\varepsilon,\ell}(\R)$ for all $t,\varepsilon>0$. A typical example is $A(u)=|u|^{\ell+1}$, $\ell\ge 1$.
For a flux function $A$ that fails to be convex, $n=1$, $S\equiv 0$, the same regularity can be obtained under some restrictive assumptions on the zeroes of $A''$, by combining results of Cheng \cite{cheng_86}  and Jabin \cite{jabin_10}.

In multiple dimensions,  or for $S$ non-smooth, these arguments do not apply anymore. In this case, the best known regularity estimates rely on the kinetic formulation of \eqref{eq:intro_SCL}, as introduced by Lions, Perthame and Tadmor in \cite{lions_perthame_tadmor_94}. In this work it was observed\footnote{In fact, \cite{lions_perthame_tadmor_94} treated the case $S\equiv0$ but the same applies to non-vanishing $S$.} that if $u$ is an entropy solution to \eqref{eq:intro_SCL} then the kinetic function 
\begin{equation}
f(t,x,v):=\one_{0<v<u(t,x)}-\one_{0>v>u(t,x)}\label{eq:intro_kin_fct}
\end{equation}
 satisfies 
\begin{equation}
\partial_{t}f+a(v)\cdot\nabla_{x}f=\partial_{v}m+\delta_{v=u}S,\label{eq:kin}
\end{equation}
for some Radon measure $m\geq0$ and $a:=A'$. Based on this and on averaging techniques, regularity estimates for bounded entropy solutions to \eqref{eq:intro_SCL} have been obtained in \cite{lions_perthame_tadmor_94} assuming a non-degeneracy property for the flux $A$ and $S\equiv0$. For the special case of \eqref{eq:intro_SCL} with $A(u)=u{}^{\ell+1}$ this leads to 
\begin{equation}
u\in W_{loc}^{s,p}((0,T)\times\R^n)\quad\forall s<\frac{1}{1+2\ell},\,p<\frac{4\ell+1}{2\ell+1}.\label{eq:intro_LPT_reg}
\end{equation}
In this work, we are particularly interested in the case $\ell>1.$ The limited regularity in \eqref{eq:intro_LPT_reg} is due to the degeneracy of the flux in $u=0$.

Motivated by some ideas going back to Tadmor and Tao \cite{tadmor_tao_07}, we introduce a new decomposition of entropy solutions $u$ which allows to make use of the fact that apart from the degeneracy at $u=0$, the flux $A(u)=u{}^{\ell+1}$ has non-vanishing second derivative. Using this aspect alone we show that it is possible to improve the regularity in \eqref{eq:intro_LPT_reg} to $s<\frac{1}{2+\ell}$. In the literature, a key draw-back of the methods to estimate the regularity of solutions to \eqref{eq:intro_LPT_reg} based on averaging techniques is that these methods are not able to make use of the sign of the entropy dissipation measure $m$ in \eqref{eq:kin}. Indeed, these arguments could only use that $m$ has locally finite mass. In contrast, we make use of the observation that for entropy solutions to \eqref{eq:intro_SCL} the entropy defect measure $m$ has, thanks to its sign, locally finite singular moments, that is, $\abs{v}^{-\gamma}m$ has locally finite mass for all $\gamma\in[0,1)$. This is, to our knowledge, the first time that a kinetic averaging lemma manages, when applied to scalar conservation laws, to take advantage of the sign of the entropy production (see also \cite{porous_medium}). Specializing our results to the particular case $A(u)=u^{l+1}$ we obtain the following result.
\begin{prop}
\label{thm:scalarconservationlaw} Let $\ell\geq1$, $u_{0}\in L^{1}(\R)$, $S\in L^{1}([0,T]\times\R)$ and $u(t,x)$ be an entropy solution of \eqref{eq:intro_SCL} with $n=1$,   $A(v)=\abs{v}^{\ell+1}$ or $A(v)=\sgn(v)\abs{v}^{\ell+1}$ and associated kinetic function $f$ as in \eqref{eq:intro_kin_fct}. Then, for all $\phi\in C_{c}^{\infty}(\R)$,
\[
\int f(t,x,v)\phi(v)\,dv\in W_{loc}^{s,1}((0,T)\times\R^{n})\quad\forall s<\min\left(\frac{1}{3},\frac{1}{\ell+1}\right).
\]
In particular, if $u_{0}\in L^{\infty}(\R^{n})$ and $S\in L^{\infty}([0,T]\times\R^{n})$ then 
\[
u\in W_{loc}^{s,1}((0,T)\times\R^{n})\quad\forall s<\min\left(\frac{1}{3},\frac{1}{\ell+1}\right).
\]
\end{prop}

\begin{rem}\label{rem:quasisol}
Solutions of \eqref{eq:intro_SCL} for which the entropy dissipation $m$ is only assumed to be a locally finite signed measure are sometimes called quasi-solutions \cite{DOW03}. For the model case in Proposition~\ref{thm:scalarconservationlaw}, the arguments in \cite{lions_perthame_tadmor_94} still apply to this larger class of solutions and provide the regularity \eqref{eq:intro_LPT_reg}. However, when $\ell$ is an integer and $S\equiv0$, Crippa, Otto and Westdickenberg obtain in \cite[Proposition~4.4]{crippa_otto_westdickenberg_08}, without using averaging lemmata, a better  order of differentiability $s<1/(2+\ell)$ which has been shown to be optimal by De Lellis and Westdickenberg \cite{DLW03}. In the case where  $A$ is convex, Golse and Perthame \cite{golse_perthame_13} provide a proof of the same regularity that could be adapted to the presence of a forcing term $S$. Our arguments yield this optimal order of differentiability  $s<1/(2+\ell)$ for quasi-solutions and for all $A$ as in Proposition~\ref{thm:scalarconservationlaw} and in the presence of the forcing term $S$.
\end{rem}

Our estimates are based on certain non-degeneracy properties of general fluxes $A.$ It is a well-known phenomenon that under suitable nonlinearity assumptions on the velocity field $a(v)$, velocity averages of $f$ solving \eqref{eq:kin} are more regular than $f$. In \cite{lions_perthame_tadmor_94}, Lions, Perthame and Tadmor use the following assumption: there exists an $\alpha\in(0,1]$ such that for every  bounded interval $I\subset\R_{v}$ and all $\delta>0$, 
\begin{equation}
\sup_{\tau^{2}+|\xi|^{2}=1}\abs{\left\lbrace v\in I\colon\abs{\tau+a(v)\cdot\xi}<\delta\right\rbrace }\lesssim\delta^{\alpha}.\label{eq:usualassumption}
\end{equation}
They prove that if \eqref{eq:usualassumption} holds and $f\in L^{p}$ solves \eqref{eq:kin} with $m\in L^{q}$ for some $p,q\in(1,2]$, then for any bump function $\phi\in C_{c}^{\infty}(I)$, the velocity averages 
\[
\bar{f}(t,x):=\int f(t,x,v)\phi(v)\,dv,
\]
satisfy 
\[
\bar{f}\in W_{loc}^{s,r}((0,T)\times\R^n),\quad\forall s<\theta=\frac{\alpha/p'}{\alpha(1/p'-1/q')+2},\;\frac{1}{r}=\frac{1-\theta}{p}+\frac{\theta}{q}.
\]
In \cite{tadmor_tao_07}, Tadmor and Tao introduce the additional assumption 
\begin{equation}
\sup\left\lbrace \abs{a'(v)\cdot\xi}\colon v\in I,\,\tau^{2}+\xi^{2}=1,\,\abs{\tau+a(v)\cdot\xi}<\delta\right\rbrace \lesssim\delta^{\mu}.\label{eq:tadmortaoassumption}
\end{equation}
They prove that if \eqref{eq:usualassumption}-\eqref{eq:tadmortaoassumption} hold, then the velocity averages satisfy 
\[
\bar{f}\in W_{loc}^{s,r},\quad \forall s<\theta'=\frac{\alpha/p'}{\alpha(1/p'-1/q')+2-\mu},\;\frac{1}{r}=\frac{1-\theta'}{p}+\frac{\theta'}{q}.
\]
However, while \eqref{eq:tadmortaoassumption} applies to certain parabolic-hyperbolic PDE with non-vanishing parabolic part (which is the main focus of  \cite{tadmor_tao_07}), in many purely hyperbolic cases of interest this additional assumption is not satisfied. As an example let us consider the velocity field $a(v)=v^{\ell}$ for some $\ell\geq1$. Then \eqref{eq:usualassumption} holds with $\alpha=\frac{1}{\ell}$ and this is used in \cite{lions_perthame_tadmor_94} to obtain that entropy solutions of \eqref{eq:intro_SCL} enjoy differentiability of order $s=1/(1+2\ell)$. On the other hand, choosing $\xi=-\tau=1/\sqrt{2}$ and $v=1$ in \eqref{eq:tadmortaoassumption} shows that one cannot do better than $\mu=0$. Hence, for $a(v)=v^{\ell}$ the result in \cite{tadmor_tao_07} can not provide any improvement on \cite{lions_perthame_tadmor_94}. 

Proposition~\ref{thm:scalarconservationlaw} will be obtained as a corollary of a general averaging lemma for the kinetic equation 
\begin{equation}
\partial_{t}f+a(v)\cdot\nabla_{x}f=\partial_{v}g+h\quad\text{on }\R_{t}\times\R_{x}^{n}\times\R_{v}.\label{eq:kinmultid}
\end{equation}
As outlined above, our argument relies on the idea underlying the assumption \eqref{eq:tadmortaoassumption} in \cite{tadmor_tao_07} but requires a finer decomposition. 

We consider a velocity field $a\in C^{1}(\mathbb{R};\R^{n})$ such that the set of degeneracy points 
\begin{equation}
Z:=\left\lbrace a'=0\right\rbrace \text{ is locally finite}\label{Zfinite}
\end{equation}
and assume that there exist $\alpha<\beta\in(0,1]$ and $\kappa,\tau\geq0$ such that for any bounded interval $I\subset\R_{v}$ and $\lambda,\delta>0$ it holds 
\begin{align}
 & \sup_{\tau^{2}+\abs{\xi}^{2}=1}\abs{\left\lbrace v\in I\colon\abs{\tau+a(v)\cdot\xi}\leq\delta\right\rbrace }\lesssim\delta^{\alpha},\label{H1}\\
 & \sup_{v\in I,\,\dist(v,Z)\leq\lambda}\abs{a'(v)}\lesssim\lambda^{\kappa},\label{H2}\\
 & \sup_{\tau^{2}+\abs{\xi}^{2}=1}\abs{\left\lbrace v\in I\colon\dist(v,Z)\geq\lambda,\,\abs{\tau+a(v)\cdot\xi}\leq\delta\right\rbrace }\lesssim\lambda^{-\tau}\delta^{\beta}.\label{H3}
\end{align}
Note that since $I$ is bounded, \eqref{H1}--\eqref{H3} are trivially satisfied for $\delta,\lambda$ large. 

Assumption \eqref{H1} is nothing but the classical nonlinearity assumption \eqref{eq:usualassumption}. Assumption \eqref{H2} is similar to the assumption \eqref{eq:tadmortaoassumption} used by Tadmor and Tao in \cite{tadmor_tao_07}, but is supposed only near the degeneracy points of $a'$. Away from the degeneracies, assumption \eqref{H3} requires the classical assumption \eqref{eq:usualassumption} to be satisfied with a better exponent $\beta>\alpha$, with a constant that blows up when approaching the degeneracies.
\begin{example}
\label{rem:H123model}Let $A\in C^{2}(I;\R^n)$ for some interval $I\subseteq\R$. 
\begin{enumerate}
\item Let $A\in C^{\infty}(I;\R)$, $n=1$. The valuation of $A$ at $v\in I$ is defined as $m_{A}(v)=\inf\{k\ge1:\,A^{(k+1)}(v)\ne0\}$, the degeneracy of $A$ on $I$ is $m_{A}:=\sup_{v\in I}m_{A}(v)$. If $0<m_{A}<\infty$ we say that $A$ is non-degenerate of order $m_{A}$. In this case \eqref{H1} is satisfied with $\alpha=1/m_{A}$ (cf.~\cite[Lemma 1]{BJ10}).
\item Let $a'$ be $\kappa$-Hölder continuous, i.e.\ $A\in C^{2+\kappa}(I;\R^n)$. Then \eqref{H2} is satisfied. 
\item Assume $n=1$, \eqref{Zfinite} and that for some $\tau\geq0$ and all $\lambda>0$ 
\[
\lambda^{\tau}\lesssim\inf_{v\in I,\,\dist(v,Z)\ge\lambda}|a'(v)|.
\]
Then $a$ satisfies \eqref{H3} with $\beta=1$. 
\item Let $A(v)=\sin(v)$ or $A(v)=\cos(v)$. Then $A$ satisfies \eqref{Zfinite}-\eqref{H3} with $\alpha=1/2$, $\beta=1$ and $\kappa=\tau=1$ (cf.\ Example \ref{ex:sin} below). 
\item Our model one-dimensional velocity field $a(v)=v^{\ell}$ satisfies \eqref{Zfinite}-\eqref{H3} with $\alpha=1/\ell$, $\beta=1$ and $\kappa=\tau=\ell-1$. 
\end{enumerate}
\end{example}
\begin{thm}
\label{thm:averaging} Let $a\in C^{1}(\R;\R^{n})$ satisfy \eqref{Zfinite}-\eqref{H3}. Let $p,q\in[1,2]$ with $p\geq q$, $\gamma\in[0,1]$ and $\sigma \in [0,1)$. Assume that $f\in L_{loc}^{p}(\R_{t}\times\R_{x}^{n};W_{loc}^{\sigma,p}(\R_{v}))$ solves the kinetic equation \eqref{eq:kinmultid} with 
\begin{equation}
h,(1+\dist(v,Z)^{-\gamma})g\in\begin{cases}
L_{loc}^{q}(\R_{t}\times\R_{x}^{n}\times\R_{v}) & \text{ if }q\in(1,p],\\
\mathcal{M}_{loc}(\R_{t}\times\R_{x}^{n}\times\R_{v}) & \text{ if }q=1.
\end{cases}\label{H4}
\end{equation}
Then, for any $\phi\in C_{c}^{\infty}(\R)$, the average $\bar{f}(t,x)=\int f(t,x,v)\phi(v)\,dv$ satisfies 
\[
\bar{f}\in W_{loc}^{s,r}(\R_{t}\times\R_{x}^{n})\qquad\forall s\in[0,s_{*}),
\]
where the order of differentiability $s_{*}$ is given by 
\begin{align*}
 & s_{*}=(1-\eta)\theta_{\alpha}+\eta\theta_{\beta},\\
 & \theta_{a}=\frac{a/\bar{p}}{a(1/\bar{p}-1/q')+2}\quad(a=\alpha,\beta),\qquad\eta=\frac{E_{1}}{E_{1}+E_{2}},\\
 & \bar{p}\in[\frac{p'}{1+\sigma p'},p']\cap(1,\infty),\\
 & E_{1}=\min\left((\kappa+\gamma),\frac{1}{\alpha}-(1-\gamma)\right)\theta_{\alpha},\\
 & E_{2}=\max\left(\frac{2\tau}{\beta}-\kappa-\gamma,\frac{\tau-1}{\beta}+1-\gamma,0\right)\theta_{\beta},
\end{align*}
and the order of integrability $r$ is given by 
\begin{align*}
\frac{1}{r}=\frac{1-\eta}{r_{\alpha}}+\frac{\eta}{r_{\beta}},\qquad\frac{1}{r_{a}}=\frac{1-\theta_{a}}{p}+\frac{\theta_{a}}{q}\quad(a=\alpha,\beta).
\end{align*}

\end{thm}
The proof of Theorem~\ref{thm:averaging} consists in splitting the velocity average into velocities which are close to the degeneracy set $\{v\in\R:\,\dist(v,Z)\leq\lambda\}$ and far away from it $\{v\in\R:\,\dist(v,Z)\geq\lambda\}$. Close to $Z$, assumption \eqref{H1} only allows us to obtain a differentiability of order $\theta_{\alpha}$ by arguing as in \cite{lions_perthame_tadmor_94}, but assumption \eqref{H2} allows us (in the spirit of \cite{tadmor_tao_07}) to estimate the corresponding norms with $\lambda^{E_{1}}$. Away from $Z$, assumption \eqref{H3} allows us to obtain differentiability of the better order $\theta_{\beta}$, with a corresponding estimate in $\lambda^{-E_{2}}$. Then optimizing the choice of $\lambda$ yields the conclusion.

Proposition \ref{thm:scalarconservationlaw} is an immediate consequence of the following result.
\begin{thm}
\label{thm:scl} Let $A\in C^{2}(\R;\R^{n})$ satisfy \eqref{Zfinite}-\eqref{H3}, $u_{0}\in L^{1}(\R_{x}^{n})$, $S\in L^{1}([0,T]\times\R_{x}^{n})$ and $u(t,x)$ be an entropy solution of \eqref{eq:intro_SCL} with associated kinetic function $f$ as in \eqref{eq:intro_kin_fct}. Then, for all $\phi\in C_{c}^{\infty}(\R)$,
\[
\int f(t,x,v)\phi(v)\,dv\in W_{loc}^{s,r}((0,T)\times\R_{x}^{n})\quad\forall s<s_{*},
\]
where 
\begin{align*}
 & s_{*}=(1-\eta)\theta_{\alpha}+\eta\theta_{\beta},\\
 & \theta_{a}=\frac{a}{a+2},\quad(a=\alpha,\beta),\quad\eta=\frac{E_{1}}{E_{1}+E_{2}},\\
 & E_{1}=\min\left(\kappa+1,\frac{1}{\alpha}\right)\theta_{\alpha},\\
 & E_{2}=\max\left(\frac{2\tau}{\beta}-\kappa-1,\frac{\tau-1}{\beta},0\right)\theta_{\beta},
\end{align*}
and the order of integrability $r$ is given by 
\begin{align*}
\frac{1}{r}=\frac{1-\eta}{r_{\alpha}}+\frac{\eta}{r_{\beta}},\qquad\frac{1}{r_{a}}=\frac{1+\theta_{a}}{2},\quad(a=\alpha,\beta).
\end{align*}
 In particular, if $u_{0}\in L^{\infty}(\R^{n})$ and $S\in L^{\infty}([0,T]\times\R^{n})$ then 
\[
u\in W_{loc}^{s,r}((0,T)\times\R_{x}^{n})\quad\forall s<s_{*}.
\]
\end{thm}
\begin{example}\label{ex:sin}
Consider \eqref{eq:intro_SCL} with flux $A(v)=\sin(v)$ or $A(v)=\cos(v)$, $u_{0}\in L^{\infty}(\R)$ and $S\in L^{\infty}([0,T]\times\R)$. Then 
\[
u\in W_{loc}^{s,r}((0,T)\times\R)\quad\forall s<\frac{1}{3},\,r\le\frac{3}{2},
\]
despite the existence of degeneracy points, i.e.~$\{v\in\R:\,A''(v)=0\}\ne\emptyset$. This improves the previously known regularity of $s<\frac{1}{5}$.
\end{example}

\subsection*{Notation}

We will use the symbol $\lesssim$ to denote inequality up to a constant that does not depend on the interpolation parameters $\lambda$ and $\delta$. Further, $\mathcal{F}=\mathcal{F}_{t,x}$ denotes the Fourier transform in the $(t,x)$ variables and for $(\tau,\xi)\in\R^{n+1}$ let 
\[
(\tau',\xi'):=\frac{1}{\sqrt{\tau^{2}+|\xi|^{2}}}(\tau,\xi),
\]
so that $(\tau')^{2}+|\xi'|{}^{2}=1$. For $p \ge 1$ we let $p'$ be the conjugate exponent, that is, $\frac{1}{p}+ \frac{1}{p'}=1$. For $Z\subseteq\R$, $\dist(v,Z) := \inf_{z\in Z} |v-z|$.

\subsection*{Structure of the paper}

The plan of the paper is as follows. In Section~\ref{s:prelim} we present the proof of the main results Theorem~\ref{thm:averaging} and Theorem \ref{thm:scl}. Some background material on scalar conservation laws with an $L^{1}$-force is presented in Appendix \ref{sec:kinetic_solutions}. In Appendix \ref{sec:A-basic-estimate} we recall a basic $L^{p}$ estimate for Fourier multipliers.

\section{Proofs of the main results\label{s:prelim}}

\subsection*{Reduction to $Z\cap\supp\phi=\lbrace0\rbrace$ and localization}

If $Z\cap\supp\phi=\emptyset$ then Theorem~\ref{thm:averaging} does not improve on \cite{lions_perthame_tadmor_94}, so we may assume that $Z\cap\supp\phi$ contains at least one element. If $Z\cap\supp\phi=\lbrace v_{1},\ldots,v_{N}\rbrace$, we may choose a smooth partition of unity $\phi_{1}(v)+\cdots+\phi_{N}(v)=1$ such that $Z\cap\supp\phi_{j}=\lbrace v_{j}\rbrace$ for all $j\in\lbrace1,\ldots,N\rbrace$. Since $\bar{f}=\bar{f}_{1}+\cdots+\bar{f}_{N}$ with $\bar{f}_{j}=\int f(t,x,v)\phi(v)\phi_{j}(v)\,dv$, it suffices to prove Theorem~\ref{thm:averaging} in the case where $Z\cap\supp\phi$ contains exactly one element. Translating $v$, we may moreover assume that this element is $0$.

Note that we may moreover assume that $f,g,h$ have compact support: for $\phi(t,x,v)$ smooth and compactly supported, the function $\tilde{f}=\phi f$ is compactly supported and satisfies 
\begin{equation}
\partial_{t}\tilde{f}+a(v)\cdot\nabla_{x}\tilde{f}=\partial_{v}\tilde{g}+\tilde{h},\label{eq:loc_kin}
\end{equation}
with 
\begin{equation*}\begin{split}
  \tilde h &= f \partial_t \phi + a(v)f\nabla_x \phi - (\partial_v \phi) g + h\phi\\
  \tilde g &= \phi g.
\end{split}\end{equation*}
We note that $\tilde{h},\tilde{g}$ are compactly supported and satisfy \eqref{H4} since $q\leq p$.

Hence, the assumptions \eqref{H2}-\eqref{H4} become 
\begin{align}
 & \sup_{v\in I,\,\abs{v}\leq\lambda}\abs{a'(v)}\lesssim\lambda^{\kappa},\label{H2:0}\\
 & \sup_{\tau^{2}+\abs{\xi}^{2}=1}\abs{\left\lbrace v\in I\colon\abs{v}\geq\lambda,\,\abs{\tau+a(v)\cdot\xi}\leq\delta\right\rbrace }\lesssim\lambda^{-\tau}\delta^{\beta},\label{H3:0}\\
 & h,(1+\abs{v}^{-\gamma})g\in\begin{cases}
L^{q}(\R_{t}\times\R_{x}^{n}\times\R_{v}) & \text{ if }q\in(1,2],\\
\mathcal{M}(\R_{t}\times\R_{x}^{n}\times\R_{v}) & \text{ if }q=1.
\end{cases}\label{H4:0}
\end{align}

\subsection*{Separating small and large velocities}

We fix a bounded interval $I\subset [-\Lambda,\Lambda]\subset\R$ and a bump function $\phi\in C_{c}^{\infty}(I)$. We further fix a cut-off function $\eta_{1}\in C_{c}^{\infty}(\R)$ satisfying 
\begin{align*}
  \eta(v) &\in [0,1] \text{ for all } v\in \R,\\
  \eta_{1}(v)&\equiv1\text{ for }\abs{v}\leq1,\quad \eta_{1}(v)\equiv0\text{ for }\abs{v}\geq2.
\end{align*}
Then we set $\eta_{2}:=1-\eta_{1}$, so that for any $\lambda>0$ it holds 
\begin{equation}
\begin{split}\label{eq:f_l_split}
\bar{f}(t,x) & =\int f(t,x,v)\phi(v)\eta_{1}(\frac{v}{\lambda})\:dv+\int f(t,x,v)\phi(v)\eta_{2}(\frac{v}{\lambda})\:dv\\
 & =:A_{1}^{\lambda}f+A_{2}^{\lambda}f.
\end{split}
\end{equation}
Note that for all $\lambda\geq\Lambda$ we have $A^\lambda_{1}f=\bar f$ and $A_2^\lambda f = 0$ so that in the sequel we will only need to consider $\lambda\leq \Lambda$.

Since $A_{2}^{\lambda}f$ does not see small velocities, we could use assumption \eqref{H3:0} and obtain from \cite{lions_perthame_tadmor_94} that $A_{2}^{\lambda}f$ has differentiability of order $s=\theta_{\beta}$. In contrast, for $A_{1}^{\lambda}f$ we can only use \eqref{H1} to see that it has differentiability of order $s=\theta_{\alpha}<\theta_{\beta}$. But our assumptions allow us to take advantage of the fact that $A_{1}^{\lambda}f$ only sees small velocities in two ways: first, by using that $a'(v)$ is small thanks to \eqref{H2:0} -- along the idea that led to introducing the assumption \eqref{eq:tadmortaoassumption} in \cite{tadmor_tao_07}; and second, by using the finite singular moment assumption \eqref{H4:0} on $g$. That way we find that the estimate for $A_{1}^{\lambda}f$ comes with a constant that goes to zero when $\lambda$ approaches zero (cf.~Lemma \ref{lem:A1f} below). On the other hand, the estimate for $A_{2}^{\lambda}f$ comes with a constant that blows up when $\lambda$ approaches zero (cf.~Lemma \ref{lem:A2f} below).
\begin{lem}
\label{lem:A1f} For all $s\in[0,\theta_{\alpha})$ there exists a constant $C>0$ such that for any $\lambda\leq \Lambda$ it holds 
\[
\norm{A_{1}^{\lambda}f}_{W^{s,r_{\alpha}}}\leq C\lambda^{E_{1}},
\]
where $E_{1}$ is given by 
\[
E_{1}=\min\left((\kappa+\gamma),\frac{1}{\alpha}-(1-\gamma)\right)\theta_{\alpha}.
\]
\end{lem}
\begin{proof}
The proof will follow the strategy of \cite[Averaging Lemma~2.1]{tadmor_tao_07}, the main difference residing in the fact that we want to keep track of the dependence on $\lambda$ of all the estimates.

We fix $\psi_{0}(z)$ supported in $\abs{z}\leq2$ and $\psi_{1}(z)$ supported in $1/2\leq\abs{z}\leq2$ such that 
\[
1\equiv\psi_{0}(z)+\sum_{j\geq1}\psi_{1}(2^{-j}z),\quad\forall z\in\C.
\]
For any $\delta>0$ we decompose $f$ as 
\begin{equation}\label{eq:f_decomp_1}
  f=f^{0}+f^{1},
\end{equation}
where 
\begin{align}
f^{0} & =\mathcal{F}^{-1}\psi_{0}\left(\frac{i\tau'+ia(v)\cdot\xi'}{\delta}\right)\mathcal{F}f,\label{eq:f_decomp}\\
f^{1} & =\sum_{j\geq1}\mathcal{F}^{-1}\psi_{1}\left(\frac{i\tau'+ia(v)\cdot\xi'}{2^{j}\delta}\right)\mathcal{F}f=\sum_{j\geq1}f^{(j)}.\nonumber 
\end{align}
Then we estimate the $L{}^{p}$ norm of $A_1^{\lambda}{f^{0}}$ and the $\dot{W}^{1,q}$ norm of $A_1^{\lambda}{f^{1}}$ and conclude using real interpolation.

We treat first the case $q>1$. Invoking Lemma~\ref{lem:Lpestim} and using \eqref{H1} we have 
\begin{align}
\norm{A_{1}^{\lambda}f^{0}}_{L^{p}} & \lesssim\sup_{\tau^{2}+\abs{\xi}^{2}=1}\abs{\left\lbrace v\in I\colon\abs{v}\leq2\lambda,\,\abs{\tau+a(v)\cdot\xi}\leq2\delta\right\rbrace }^{1/\bar{p}}\nonumber \\
 & \lesssim\min(\delta^{\alpha},\lambda)^{1/\bar{p}}.\label{estim:A1f0}
\end{align}
Using \eqref{eq:kin} in Fourier variables yields, for all $(\tau,\xi,v) \in \R^{2+n}$ such that $\tau'+i\xi'\cdot a(v)\ne 0$,
\begin{align*}
  \mathcal{F} f 
  &= \frac{1}{|(\tau,\xi)|}\frac{1}{i\tau'+i\xi'\cdot a(v)} \mathcal{F} (\partial_vg+h)\\
  &= \mathcal{F} (-\Delta_{t,x})^{-1/2} \mathcal{F}^{-1} \frac{1}{i\tau'+i\xi'\cdot a(v)} \mathcal{F} (\partial_vg+h).
\end{align*}
Hence, setting $\widetilde{\psi}_{1}(z):=\psi_{1}(z)/z$ we find that for $j\geq1$ we have
\begin{align*}
  \mathcal{F}(-\Delta_{t,x})^{1/2}A_{1}^{\lambda}f^{(j)} &=\frac{1}{2^{j}\delta}\int\widetilde{\psi}_{1}\left(\frac{i\tau'+i a(v)\cdot\xi'}{2^{j}\delta}\right)\mathcal{F}\partial_{v}g\,\phi(v)\eta_{1}\left(\frac{v}{\lambda}\right)\,dv\\
 & \quad+\frac{1}{2^{j}\delta}\int\widetilde{\psi}_{1}\left(\frac{i\tau'+i a(v)\cdot\xi'}{2^{j}\delta}\right)\mathcal{F}h\,\phi(v)\eta_{1}\left(\frac{v}{\lambda}\right)\,dv.
\end{align*}
Integrating by parts thus yields
\begin{align*}
 & \mathcal{F}(-\Delta_{t,x})^{1/2}A_{1}^{\lambda}f^{(j)}\\
 & =-\frac{1}{(2^{j}\delta)^{2}}\int\widetilde{\psi}_{1}'\left(\frac{i\tau'+i a(v)\cdot\xi'}{2^{j}\delta}\right)ia'(v)|v|^\gamma\cdot\xi'\mathcal{F}|v|^{-\gamma} g\,\phi(v)\eta_{1}\left(\frac{v}{\lambda}\right)\,dv\\
 & \quad-\frac{1}{2^{j}\delta\lambda}\int\widetilde{\psi}_{1}\left(\frac{i\tau'+i a(v)\xi'}{2^{j}\delta}\right)|v|^{\gamma}\mathcal{F}|v|^{-\gamma}g\,\phi(v)\eta_{1}'\left(\frac{v}{\lambda}\right)\,dv\\
 & \quad-\frac{1}{2^{j}\delta}\int\widetilde{\psi}_{1}\left(\frac{i\tau'+i a(v)\xi'}{2^{j}\delta}\right)|v|^{\gamma}\mathcal{F}|v|^{-\gamma}g\,\phi'(v)\eta_{1}\left(\frac{v}{\lambda}\right)\,dv\\
 & \quad+\frac{1}{2^{j}\delta}\int\widetilde{\psi}_{1}\left(\frac{i\tau'+i a(v)\cdot\xi'}{2^{j}\delta}\right)\mathcal{F}h\,\phi(v)\eta_{1}\left(\frac{v}{\lambda}\right)\,dv.
\end{align*}
Invoking Lemma~\ref{lem:Lpestim} with $p=q, \sigma=0, r=q'$, recalling that $\xi'$ is a bounded $L^{q}$ multiplier, that $\abs{v}^{-\gamma}g\in L^{q}$ and using \eqref{H2:0}, we deduce 
\begin{align*}
\norm{A_{1}^{\lambda}f^{(j)}}_{\dot{W}^{1,q}} & \lesssim2^{-2j}\delta^{-2}\lambda^{\kappa+\gamma}(2^{j}\delta)^{\frac{\alpha}{q'}}+2^{-j}\delta^{-1}\lambda^{\gamma-1}(2^{j}\delta)^{\frac{\alpha}{q'}}\\
 & \quad+2^{-j}\delta^{-1}\lambda^{\gamma}(2^{j}\delta)^{\frac{\alpha}{q'}}+2^{-j}\delta^{-1}(2^{j}\delta)^{\frac{\alpha}{q'}}\\
 & \lesssim2^{-2j}\delta^{-2}\lambda^{\kappa+\gamma}(2^{j}\delta)^{\frac{\alpha}{q'}}+2^{-j}\delta^{-1}\lambda^{\gamma-1}(2^{j}\delta)^{\frac{\alpha}{q'}}.
\end{align*}
In the second inequality we were able to discard the two last terms in the previous line because $\lambda\lesssim1$ and $\gamma\leq1$. Since $\alpha < q'$, summing over $j\geq1$ yields 
\begin{equation}
\norm{A_{1}^{\lambda}f^{1}}_{\dot{W}^{1,q}}\lesssim\delta^{-2+\frac{\alpha}{q'}}\lambda^{\kappa+\gamma}+\delta^{-1+\frac{\alpha}{q'}}\lambda^{\gamma-1}.\label{estim:A1f1}
\end{equation}
From \eqref{estim:A1f0}-\eqref{estim:A1f1} we obtain for all $t>0$ that 
\begin{align*}
K(t,A_{1}^{\lambda}f) & :=\inf_{A^\lambda_{1}f=\tilde{f}^{0}+\tilde{f}^{1}}\left(\norm{\tilde{f}^{0}}_{L^{p}}+t\norm{\tilde{f}^{1}}_{\dot{W}^{1,q}}\right)\\
 & \lesssim\delta^{\frac{\alpha}{\bar{p}}}+t\delta^{\frac{\alpha}{q'}-2}\lambda^{\kappa+\gamma}+t\delta^{\frac{\alpha}{q'}-1}\lambda^{\gamma-1}.
\end{align*}
Next we optimize in $\delta$. We choose it of the form $\delta=t^{a}\lambda^{b}$, where $b$ will be chosen later and $a$ is determined by balancing the powers of $t$ in the first two terms: 
\[
a\frac{\alpha}{\bar{p}}=1+a\left(\frac{\alpha}{q'}-2\right)\quad\text{i.e.}\quad a=\frac{\bar{p}}{\alpha}\theta_{\alpha}.
\]
This gives 
\begin{align*}
t^{-\theta_{\alpha}}K(t,A_{1}^{\lambda}f)\lesssim\lambda^{b\frac{\alpha}{\bar{p}}}+\lambda^{b\left(\frac{\alpha}{q'}-2\right)+\kappa+\gamma}+t^{\frac{\bar{p}}{\alpha}\theta_{\alpha}}\lambda^{b\left(\frac{\alpha}{q'}-1\right)+\gamma-1}.
\end{align*}
Note that the last term is small for small $t$. On the other hand for large $t$ we can use the fact (obtained from \eqref{estim:A1f0} by sending $\delta\to\infty$) that 
\[
\norm{A_{1}^{\lambda}f}_{L^{p}}\lesssim\lambda^{\frac{1}{\bar{p}}},
\]
to deduce, for any $\mu>0$, 
\[
t^{-\theta_{\alpha}}K(t,A_{1}^{\lambda}f)\lesssim\begin{cases}
\lambda^{b\frac{\alpha}{\bar{p}}}+\lambda^{b\left(\frac{\alpha}{q'}-2\right)+\kappa+\gamma}+\mu^{\frac{\bar{p}}{\alpha}\theta_{\alpha}}\lambda^{b\left(\frac{\alpha}{q'}-1\right)+\gamma-1} & \text{ for }t\leq\mu,\\
\mu^{-\theta_{\alpha}}\lambda^{\frac{1}{\bar{p}}} & \text{ for }t\geq\mu.
\end{cases}
\]
Next we choose $\mu$ in order to balance the last terms of the above two lines, i.e. 
\[
\mu=\lambda^{\frac{\alpha}{\alpha+\bar{p}}\frac{1}{\theta_{\alpha}}\left(\frac{1}{\bar{p}}+b\left(1-\frac{\alpha}{q'}\right)+1-\gamma\right)},
\]
and conclude that 
\[
t^{-\theta_{\alpha}}K(t,A_{1}^{\lambda}f)\lesssim\lambda^{b\frac{\alpha}{\bar{p}}}+\lambda^{b\left(\frac{\alpha}{q'}-2\right)+\kappa+\gamma}+\lambda^{\frac{1}{\alpha+\bar{p}}\left(1-\alpha(1-\gamma)-b\alpha\left(1-\frac{\alpha}{q'}\right)\right)}.
\]
Finally we want to choose $b$ to optimize the above powers of $\lambda$ : set 
\[
E_{1}:=\sup_{b\in\R}\;\min\left\lbrace \begin{gathered}\frac{\alpha}{\bar{p}}b\\
\kappa+\gamma-\left(2-\frac{\alpha}{q'}\right)b\\
\frac{1}{\alpha+\bar{p}}\left(1-\alpha(1-\gamma)-\alpha\left(1-\frac{\alpha}{q'}\right)b\right)
\end{gathered}
\right\rbrace .
\]
We denote by $L_{1}(b),L_{2}(b),L_{3}(b)$ the three affine functions of $b$ appearing in the definition of $E_{1}$. Since $L_{1}$ is increasing and $L_{2},L_{3}$ are decreasing, $E_{1}$ is given by 
\begin{align*}
E_{1} & =\min(L_{1}(L_{1}=L_{2}),L_{1}(L_{1}=L_{3}))\\
 & =\min\left((\kappa+\gamma),\frac{1}{\alpha}-(1-\gamma)\right)\theta_{\alpha}.
\end{align*}

Then, denoting by $\norm{\cdot}_{\theta}$ the norm in the real interpolation space $[L^{p},\dot{W}^{1,q}]_{\theta,\infty}$ (see e.g. \cite{bergh_lofstrom} for definition and properties), we have 
\[
\norm{A_{1}^{\lambda}f}_{\theta_{\alpha}}\lesssim\lambda^{E_{1}},
\]
which implies the conclusion of Lemma~\ref{lem:A1f}.

In the case $q=1$, we obtain the same estimates, but the space $\dot{W}^{1,q}=(-\Delta_{t,x})^{-1/2}L^{q}$ has to be replaced with $(-\Delta_{t,x})^{-1/2}\mathcal{M}$. Since this space contains $\dot{W}^{s,1}$ for all $s<1$ we still obtain the conclusion. \end{proof}
\begin{lem}
\label{lem:A2f} For all $s\in[0,\theta_{\beta})$ there exists $C>0$ such that for any $\lambda\leq \Lambda$ it holds 
\[
\norm{A_{2}^{\lambda}f}_{W^{s,r_{\beta}}}\leq C\lambda^{-E_{2}},
\]
where $E_{2}$ is given by 
\[
E_{2}=\max\left(\frac{2\tau}{\beta}-\kappa-\gamma,\frac{\tau-1}{\beta}+1-\gamma,0\right)\theta_{\beta}.
\]
\end{lem}
\begin{proof}
As in the proof of Lemma~\ref{lem:A1f} we consider the decomposition \eqref{eq:f_decomp_1} and treat first the case $q>1$.

Let $\widetilde{\eta}(v)=\eta_{1}(v/2)-\eta_{1}(v)$, so that $\widetilde{\eta}$ is supported inside $\lbrace1\leq\abs{v}\leq4\rbrace$ and $\eta_{2}(v)=\sum_{k\geq0}\widetilde{\eta}(v/2^{k})$. Hence, it holds 
\begin{align*}
A_{2}^{\lambda}f & =\sum_{k\geq0}A_{2}^{(k)}f,\\
A_{2}^{(k)}f & =\int f(t,x,v)\phi(v)\widetilde{\eta}\left(\frac{v}{2^{k}\lambda}\right)\,dv.
\end{align*}
Next we estimate $\norm{A_{2}^{(k)}f}_{\theta_{\beta}}$. Fix $k\geq0$ and let $\mu:=2^{k}\lambda$, so that 
\begin{align*}
A_{2}^{(k)}f
&=\int f(t,x,v)\phi(v)\widetilde{\eta}(v/\mu)\,dv\\
&=\sum_{j\ge 1} \int f^{(j)}(t,x,v)\phi(v)\widetilde{\eta}(v/\mu)\,dv=:\sum_{j\ge 1} A_{2}^{(k)}f^{(j)},
\end{align*}
with $f^{(j)}$ defined as in \eqref{eq:f_decomp}. Analogously,
\begin{align*}
A_{2}^{(k)}f^0
&=\int f^0(t,x,v)\phi(v)\widetilde{\eta}(v/\mu)\,dv.
\end{align*}

Note that $A_{2}^{(k)}f$ is nonzero only for $k$ such that $\mu=2^{k}\lambda\lesssim1$, since $\phi$ is supported in a compact interval and $\widetilde{\eta}(v)$ vanishes for $\abs{v}\leq1$. By Lemma~\ref{lem:Lpestim} and assumption \eqref{H3:0} it holds 
\begin{align}
\norm{A_{2}^{(k)}f^{0}}_{L^{p}} & \lesssim\sup_{\tau^{2}+\xi^{2}=1}\abs{\left\lbrace v\in I\colon4\mu\geq\abs{v}\geq\mu,\,\abs{\tau+a(v)\xi}\leq2\delta\right\rbrace }^{1/\bar{p}}\nonumber \\
 & \lesssim\min(\mu^{-\tau}\delta^{\beta},\mu)^{\frac{1}{\bar{p}}}.\label{estim:A2f0}
\end{align}
As in the proof of Lemma~\ref{lem:A1f} we have 
\begin{align*}
 & \mathcal{F}(-\Delta_{t,x})^{1/2}A_{2}^{(k)}f^{(j)}\\
 & =-\frac{1}{(2^{j}\delta)^{2}}\int\widetilde{\psi}_{1}'\left(\frac{i\tau'+i a(v)\cdot\xi'}{2^{j}\delta}\right)ia'(v)|v|^{\gamma}\cdot\xi'\mathcal{F}|v|^{-\gamma}g\,\phi(v)\widetilde{\eta}\left(\frac{v}{\mu}\right)\,dv\\
 & \quad-\frac{1}{2^{j}\delta\mu}\int\widetilde{\psi}_{1}\left(\frac{i\tau'+i a(v)\cdot\xi'}{2^{j}\delta}\right)|v|^{\gamma}\mathcal{F}|v|^{-\gamma}g\,\phi(v)\widetilde{\eta}'\left(\frac{v}{\mu}\right)\,dv\\
 & \quad-\frac{1}{2^{j}\delta}\int\widetilde{\psi}_{1}\left(\frac{i\tau'+i a(v)\cdot\xi'}{2^{j}\delta}\right)|v|^{\gamma}\mathcal{F}|v|^{-\gamma}g\,\phi'(v)\widetilde{\eta}\left(\frac{v}{\mu}\right)\,dv\\
 & \quad+\frac{1}{2^{j}\delta}\int\widetilde{\psi}_{1}\left(\frac{i\tau'+i a(v)\cdot\xi'}{2^{j}\delta}\right)\mathcal{F}h\,\phi(v)\widetilde{\eta}\left(\frac{v}{\mu}\right)\,dv
\end{align*}
which yields, using assumptions \eqref{H2:0}-\eqref{H4:0}, 
\begin{equation}
\norm{A_{2}^{(k)}f^{1}}_{\dot{W}^{1,q}}\lesssim\delta^{-2+\frac{\beta}{q'}}\mu^{-\frac{\tau}{q'}+\kappa+\gamma}+\delta^{-1+\frac{\beta}{q'}}\mu^{-1-\frac{\tau}{q'}+\gamma}.\label{estim:A2f1}
\end{equation}
Here as in the proof of Lemma~\ref{lem:A1f} we estimated the third and fourth term by the second term on the right hand side since they come with higher powers of $\mu\lesssim1$. The estimates \eqref{estim:A2f0}-\eqref{estim:A2f1} then imply
\begin{align*}
K(t,A_{2}^{(k)}f) & :=\inf_{A_{2}^{(k)}f=\tilde{f}^{0}+\tilde{f}^{1}}\left(\norm{\tilde{f}^{0}}_{L^{p}}+t\norm{\tilde{f}^{1}}_{\dot{W}^{1,q}}\right)\\
 & \lesssim \mu^{-\frac{\tau}{\bar p}}
 \delta^{\frac{\beta}{\bar{p}}}+t\delta^{\frac{\beta}{q'}-2}\mu^{-\frac{\tau}{q'}+\kappa+\gamma}+t\delta^{\frac{\beta}{q'}-1}\mu^{-1-\frac{\tau}{q'}+\gamma}.
\end{align*}
Equilibrating the first and the second term yields the choice $\delta=t^{\frac{\bar p}{\beta}\theta_{\beta}}\mu^{b}$. Since $\|A_{2}^{(k)}f\|_{L^p} \le \mu^\frac{1}{\bar p}$ we also have $K(t,A_{2}^{(k)}f) \lesssim 1$ for all $t\ge 0$. We thus obtain
\begin{align*}
 & t^{-\theta_{\beta}}K(t,A_{2}^{(k)}f)\\
 & \lesssim\begin{cases}
\mu^{-\frac{\tau}{\bar{p}}+b\frac{\beta}{\bar{p}}}+\mu^{-\frac{\tau}{q'}+\kappa+\gamma-b\left(2-\frac{\beta}{q'}\right)}+\nu^{\frac{\bar{p}}{\beta}\theta_{\beta}}\mu^{-1-\frac{\tau}{q'}+\gamma-b\left(1-\frac{\beta}{q'}\right)} & \text{ for }t\leq\nu,\\
\nu^{-\theta_{\beta}}\mu^{\frac{1}{\bar{p}}} & \text{ for }t\geq\nu.
\end{cases}
\end{align*}
We choose $\nu=\mu^{\frac{1}{\theta_{\beta}}\frac{\beta}{\beta+\bar{p}}\left(1-\gamma+\frac{\tau}{q'}+b\left(1-\frac{\beta}{q'}\right)+\frac{1}{\bar{p}}\right)}$ to deduce 
\begin{align*}
t^{-\theta_{\beta}}K(t,A_{2}^{(k)}f) & \lesssim\mu^{-\frac{\tau}{\bar{p}}+b\frac{\beta}{\bar{p}}}+\mu^{-\frac{\tau}{q'}+\kappa+\gamma-b\left(2-\frac{\beta}{q'}\right)}\\
 & +\mu^{-\frac{1}{\beta+\bar{p}}\left(-1+\beta\left(1-\gamma+\frac{\tau}{q'}\right)+\beta\left(1-\frac{\beta}{q'}\right)b\right)}.
\end{align*}
Then optimizing in $b$ we set 
\[
E=\inf_{b\in\R}\max\left\lbrace \begin{gathered}\frac{\tau}{\bar{p}}-\frac{\beta}{\bar{p}}b\\
\frac{\tau}{q'}-\kappa-\gamma+\left(2-\frac{\beta}{q'}\right)b\\
\frac{1}{\beta+\bar{p}}\left(-1+\beta\left(1-\gamma+\frac{\tau}{q'}\right)+\beta\left(1-\frac{\beta}{q'}\right)b\right)
\end{gathered}
\right\rbrace ,
\]
and obtain (recall $\mu\lesssim1$) that 
\begin{equation}
\norm{A_{2}^{(k)}f}_{\theta_{\beta}}\lesssim\mu^{-E}=2^{-kE}\lambda^{-E}.\label{eq:A2kf}
\end{equation}
We denote by $L_{1}(b),L_{2}(b),L_{3}(b)$ the three affine functions of $b$ appearing in the definition of $E$. Since $L_{1}$ is increasing while $L_{2},L_{3}$ are increasing, it holds 
\begin{align*}
E & =\max(L_{1}(L_{1}=L_{2}),L_{1}(L_{1}=L_{3}))\\
 & =\max\left(\frac{2\tau}{\beta}-\kappa-\gamma,\frac{\tau-1}{\beta}+1-\gamma\right)\theta_{\beta}.
\end{align*}

If $E>0$, then summing \eqref{eq:A2kf} over $k\geq0$ yields $\norm{A_{2}^{\lambda}f}_{\theta_{\beta}}\lesssim\lambda^{-E}$. If $E\leq0$, then summing \eqref{eq:A2kf} over those $k$ satisfying $2^{k}\lambda\lesssim1$ yields 
\[
\norm{A_{\beta}^\lambda f}_{\theta_{2}}\lesssim\lambda^{-E}\sum_{0\leq k\leq\log(C/\lambda)}(2^{-E})^{k}\lesssim\lambda^{-E}2^{E\log(\lambda/C)}\lesssim1.
\]
Hence we conclude that $\norm{A_{2}^{\lambda}f}_{\theta_{\beta}}\lesssim\lambda^{-\max(E,0)}$.

To treat the case $q=1$ we argue as in the proof of Lemma~\ref{lem:A1f}. 
\end{proof}

\subsection*{Proofs of Theorem~\ref{thm:averaging} and Theorem~\ref{thm:scl}}

\label{s:proofs}
\begin{proof}
[Proof of Theorem~\ref{thm:averaging}] By \eqref{eq:f_l_split}, Lemma \ref{lem:A1f} and Lemma \ref{lem:A2f}, for $\lambda \lesssim 1$ and $t\ge0$,
\begin{align*}
K(t,\bar f) 
& :=\inf_{\bar  f=\tilde{f}^{0}+\tilde{f}^{1}}\left(\norm{\tilde{ f}^{0}}_{\theta_\alpha}+t\norm{\tilde{f}^{1}}_{\theta_\beta}\right)\\
&\le \norm{A_{1}^{\lambda}f}_{\theta_{\alpha}}+t\norm{A_{2}^{\lambda}f}_{\theta_{\beta}}\\
&\lesssim \lambda^{E_1}+t\lambda^{-E_2}.
\end{align*}
Choosing, $\lambda=t^{\frac{1}{E_{1}+E_{2}}}$ yields
\begin{align*}
K(t,\bar f) \lesssim t^{\frac{E_{1}}{E_{1}+E_{2}}}=t^{\eta}\qquad\forall t\lesssim 1.
\end{align*}
Since $\norm{\bar f}_{\theta_{\alpha}}\lesssim1$ (as can be seen e.g. by choosing $\lambda=\Lambda$ in Lemma~\ref{lem:A1f}) we have
\begin{align*}
  K(t,\bar f) \lesssim \norm{\bar f}_{\theta_{\alpha}} \lesssim 1\quad\forall t\ge 0.
\end{align*}
Hence, $\bar{f}$ belongs to the real interpolation space 
\[
\left[[L^{p},\dot{W}^{1,q}]_{\theta_{1},\infty},[L^{p},\dot{W}^{1,q}]_{\theta_{2},\infty}\right]_{\eta,\infty}=[L^{p},\dot{W}^{1,q}]_{\theta,\infty},
\]
where $\theta=(1-\eta)\theta_{\alpha}+\eta\theta_{\beta}$ and the equality follows from the reiteration Theorem of real interpolation. We further note that this space contains $W^{s,r}$ for all $s<s_{*}=\theta$. This argument works for $q>1$ and for $q=1$ we may adapt it as in the proof of Lemma~\ref{lem:A1f}. 
\end{proof}

\begin{proof}
[Proof of Theorem~\ref{thm:scl}] We apply the kinetic formulation for \eqref{eq:intro_SCL} (cf.~Appendix \ref{sec:kinetic_solutions}), that is, 
\[
f=\one_{0<v<u(t,x)}-\one_{0>v>u(t,x)},
\]
satisfies, in the sense of distributions, 
\begin{equation}
\partial_{t}f+a(v)\cdot\nabla_{x}f=\partial_{v}m+\delta_{v=u}S \quad\text{on } [0,T]\times\R^n_x\times\R_v\label{eq:kin_form}
\end{equation}
for some Radon measure $m\geq0$. We further note that 
$$f\in L^{1}([0,T]\times\R^n_x\times\R_v)\cap L^{\infty}([0,T]\times\R^n_x\times\R_v)$$
and $f\in L^{\infty}([0,T]\times\R^n_x;BV(\R_v)).$ Hence, by interpolation, 
  $$f\in L^{2}_{loc}([0,T]\times\R^n_x;W^{\sigma,2}(\R_v))$$ 
for all $\sigma\in[0,\frac{1}{2})$. 

For a bounded interval $I\subseteq \R_v$ let $Z\cap I =\{z_1,\dots,z_N\}$. By Proposition \ref{prop:mass_bound} below, $\abs{v-z_i}^{\alpha-1}m$ has locally finite mass for every $\alpha\in(0,1)$ and $i\in 1,\dots,N$. It follows that $\dist(v,Z)^{-\gamma}m$ has locally finite mass. 

Let $\eta \in C^\infty_c(0,T)$. Then $\tilde f := f \eta$ satisfies \eqref{eq:kinmultid} with $g=m\eta$, $h=\delta_{v=u}S\eta+\dot{\eta}f$ and  $\tilde f\in L^{2}_{loc}(\R_t\times\R^n_x;W^{\sigma,2}(\R_v))$ for all $\sigma\in[0,\frac{1}{2})$.

We now apply Theorem \ref{thm:averaging} with $p=2$, $\sigma\approx\frac{1}{2}$, $\gamma\approx1$, $q=1$, $\bar{p}\approx1$, to obtain, for all $\phi\in C_{c}^{\infty}(\R)$, 
\[
\int \tilde f\phi\,dv\in W_{loc}^{s,r}
\]
for all $s<s_{*}$, which implies the claim.
\end{proof}
\appendix

\section{Kinetic solutions for scalar conservation laws with a force\label{sec:kinetic_solutions}}

In this section we present some brief comments on the extension of the concept of kinetic solutions and their well-posedness for scalar conservation laws with an $L^{1}$-force \eqref{eq:intro_SCL}. This proceeds along the lines of \cite{CP03,P02}. We will refer to kinetic solutions also as entropy solutions. 
\begin{defn}
A kinetic/entropy solution to \eqref{eq:intro_SCL} is a function $u\in C([0,T];L^{1}(\R^{n}))$ such that the corresponding kinetic function 
\begin{equation*}
f(t,x,v)=\chi(v,u(t,x))=\one_{0<v<u(t,x)}-\one_{0>v>u(t,x)}
\end{equation*}
satisfies, in the sense of distributions, 
\begin{align}
\partial_{t}f+a(v)\cdot\nabla_{x}f & =\partial_{v}m+\delta_{u=v}S\quad\text{on }(0,T)\times\R^{n}\label{eq:app_loc_kin}\\
f_{|t=0} & =\chi(v,u_{0})\quad\text{on }\R^{n},\nonumber 
\end{align}
where $a:=A'$, $m$ is a non-negative Radon measure and 
\[
\int m(t,x,v)\,dtdx\le\mu(v)\in L_{0}^{\infty}(\R),
\]
where $L_{0}^{\infty}(\R)$ denotes the space of all essentially bounded functions decaying to zero for $|v|\to\infty$.
\end{defn}
\begin{rem}
The renormalized entropy solutions introduced in \cite{benilan_carillo_wittbold_00} provide another possible extension of entropy solutions to this $L^1$ setting, and it is very likely that they coincide with kinetic solutions.
\end{rem}
\begin{thm}
Let $u_{0}\in L^{1}(\R^{n})$, $S\in L^{1}([0,T]\times\R^{n}).$ Then there is a unique kinetic solution $u$ to \eqref{eq:intro_SCL}. For two kinetic solutions $u^{1}$, $u^{2}$ with initial conditions $u_{0}^{1},u_{0}^{2}$ and forces $S^1, S^2$ respectively we have 
\begin{equation}\label{eq:contraction}
\begin{aligned}
\sup_{t\in[0,T]}\|(u^{1}(t)-u^{2}(t))_+\|_{L^{1}(\R^{n})}&\le\|(u_{0}^{1}-u_{0}^{2})_+\|_{L^{1}(\R^{n})}\\
&\quad +\|S^1-S^2\|_{L^{1}([0,T]\times\R^{n})}.
\end{aligned}
\end{equation}
\end{thm}
\begin{proof}
\textit{Contraction}: We first note that the function $g(t,x,v)=\one_{v<u(t,x)}$ satisfies
 the same kinetic equation as $f$, since
\begin{align*}
f(t,x,v)-g(t,x,v)&=-\one_{v>0}-\one_{v=0}\one_{u(t,x)\geq 0}-\one_{v<0}\one_{u(t,x)=v},\\
(\partial_t + a(v)\cdot\nabla_x)\one_{v<0} &=0\qquad\text{in }\mathcal D'_{t,x,v},\\
\text{and }\one_{v=0}&=
\one_{u(t,x)=v}=0\qquad\text{for a.e. }(t,x,v).
\end{align*}
The proof of the contraction inequality \eqref{eq:contraction} relies on the identity
\begin{equation*}
\int_v g^1 (1-g^2)  =(u^1-u^2)_+.
\end{equation*}
We introduce nonnegative mollifiers $\Phi_\e(t,x)$
and let the subscript $\e$ denote the convolution in $(t,x)$ with $\Phi_\e$. In particular, we have
\begin{equation*}
(\partial_t + a(v)\cdot\nabla_x)g_\e=\partial_v m_\e +\left(\delta_{v=u(t,x)}S(t,x)\right)_\e,
\end{equation*}
where $\left(\delta_{v=u(t,x)}S(t,x)\right)_\e$ is the distribution given by
\begin{align*}
&\left\langle \left(\delta_{v=u(t,x)}S(t,x)\right)_\e,\theta(t,x,v)\right\rangle\\& =\int_{t,x}  S(t,x)\int_{s,y}  \theta(s,y,u(t,x))\Phi_\e(s-t,y-x).
\end{align*}
We also introduce a nonnegative  cut-off function $\chi(v)$. By dominated differentiation, for any $\e_1,\e_2>0$ we have
\begin{align*}
T&:=\partial_t\int_v g^1_{\e_1}(1-g^2_{\e_2})\chi(v) +\nabla_x\cdot\int_{v} g^1_{\e_1}(1-g^2_{\e_2})\chi(v) a(v) \\
& =\int_v \chi(v) (1-g^2_{\e_2})(\partial_t + a(v)\cdot\nabla_x)g^1_{\e_1}  \\&\quad + \int_v \chi(v) g^1_{\e_1} (\partial_t + a(v)\cdot\nabla_x)(1-g^2_{\e_2})\\
& =\lim_{\delta\to 0}\left( T^1_\delta + T^2_\delta\right),
\end{align*}
where
\begin{align*}
T^1_\delta(t,x) &=\int_{v,w} \chi(v) (1-g^2_{\e_2}(t,x,w))(\partial_t + a(v)\cdot\nabla_x)g^1_{\e_1}(t,x,v)\rho_\delta(v-w),  \\
T^2_\delta(t,x) &= \int_{v,w} \chi(v) g^1_{\e_1}(t,x,w) (\partial_t + a(v)\cdot\nabla_x)(1-g^2_{\e_2}(t,x,v))\rho_\delta(v-w),
\end{align*}
and $\rho_\delta(v)$ is an even nonnegative mollifier.
Using the equation satisfied by $g^1$ we have, for any nonnegative test function $\theta(t,x)$,
\begin{align*}
&\left\langle T_\delta^1,\theta\right\rangle \\
&=- \int_v m_{\e_1}^1(dt,dx,dv)\theta(t,x) \chi'(v)\int_w  (1-g_{\e_2}^2(t,x,w))\rho_\delta(v-w) \\
&\quad - \int_v m_{\e_1}^1(dt,dx,dv)\theta(t,x) \chi(v) \int_w  (1-g_{\e_2}^2(t,x,w))(\rho_\delta)'(v-w)\\
&\quad +
\int_{t,x}  \theta(t,x)\int_{s,y}  S^1(s,y) \Phi_{\e_1}(t-s,x-y)\chi(u^1(s,y))\\
&\qquad\cdot\int_w  (1-g_{\e_2}^2(t,x,w))\rho_\delta(u^1(s,y)-w).
\end{align*}
The second term on right-hand side is nonpositive since $w \mapsto (1-g^2_{\varepsilon_2}(t,x,w))=(\one_{w\geq u^2(t,x)})_{\e_2}$ is nondecreasing.
Moreover, since $\rho_\delta$ is even, for any $(t,x,v)$ we have as $\delta\to 0$, 
\begin{align*}
&\int dw\, (1-g_{\e_2}^2(t,x,w))\rho_\delta(v-w) \\&=
\int ds dy\, \Phi_{\e_2}(t-s,x-y) \int dw\, \one_{w\geq u^2(s,y)}\rho_\delta(v-w) \\
&\to \int ds dy\, \Phi_{\e_2}(t-s,x-y)\sgn^+_{\frac 12}(v-u^2(s,y))\\
& = [\sgn_{\frac 12}^+(v-u^2)]_{\e_2}(t,x),
\end{align*}
where $\sgn_{\frac 12}^+(z)=\one_{(0,\infty)}(z)+\frac{1}{2}\one_{\{0\}}(z)$. Hence, we find
\begin{align*}
\limsup_{\delta\to 0}\left\langle T_\delta^1,\theta\right\rangle 
&\leq  \int m_{\e_1}^1(dt,dx,dv)\theta(t,x) \vert\chi'(v)\vert \\
& \quad +
\int_{t,x}  \theta(t,x)\int_{s,y} S^1(s,y) \Phi_{\e_1}(t-s,x-y)\chi(u^1(s,y))\\
&\qquad\cdot[\sgn_{\frac 12}^+(u^1(s,y)-u^2)]_{\e_2}(t,x).
\end{align*}
A similar computation shows
\begin{align*}
\limsup_{\delta\to 0}\left\langle T_\delta^2,\theta\right\rangle 
&\leq  \int m_{\e_2}^2(dt,dx,dv)\theta(t,x) \vert\chi'(v)\vert \\
& \quad -\int_{t,x}  \theta(t,x)\int_{s,y} S^2(s,y) \Phi_{\e_2}(t-s,x-y)\chi(u^2(s,y))\\
&\qquad\cdot[\sgn_{\frac 12}^+(u^1-u^2(s,y))]_{\e_1}(t,x).
\end{align*}
By Fatou's lemma these inequalities imply 
\begin{align*}
\left\langle T,\theta\right\rangle
&\leq  \big\langle\int  m^1_{\e_1}(\cdot,\cdot,dv)\vert\chi'(v)\vert,\theta\big\rangle +\big\langle\int m^2_{\e_2}(\cdot,\cdot,dv)\vert\chi'(v)\vert,\theta\big\rangle \\
&\quad +\big\langle \big(S^1\chi(u^1)\sgn_{\frac 12}^+(u^1-u^2)_{\e_2}\big)_{\e_1},\theta\big\rangle\\
&\quad- \big\langle \big(S^2 \chi(u^2)\sgn_{\frac 12}^+(u^1-u^2)_{\e_1}\big)_{\e_2},\theta\big\rangle.
\end{align*}
Next we \enquote{integrate} this inequality in $x$, that is, we apply it to a test function $\theta(t,x)=\zeta(t)K(x)\geq 0$ and let $K(x)$ approach $K\equiv 1$. Note that since
\begin{equation*}
\int \vert g^1(s,y,v)\vert \cdot \vert 1-g^2(s',y',v)\vert\, dv =(u^1(s,y)-u^2(s',y'))_+,
\end{equation*}
for any $\theta(t,x)\in L^\infty$ and $\psi(v)\in L^\infty_{loc}$ we have
\begin{align*}
&\int_{t,x,v} \left\vert g^1_{\e_1}(1-g^2_{\e_2}) \theta(t,x)\psi(v) \chi(v)\right\vert  \\
&\leq \Vert\psi\Vert_{L^\infty(\supp\chi)}\Vert \theta\Vert_{L^\infty} \left( \Vert u^1\Vert_{L^1_{t,x}} + \Vert u^2\Vert_{L^1_{t,x}}\right).
\end{align*}
Using this together with $S^j\in L^1$ and $\int m^j(dt,dx,dv)\vert\chi'(v)\vert <\infty$, and letting $K(x)$ approach $K\equiv 1$ nicely enough, we  obtain 
\begin{align*}
&\partial_t\int_{x,v} g^1_{\e_1}(1-g^2_{\e_2})\chi(v) \\
 &\leq \int m^1_{\e_1}(\cdot,dx,dv)\vert\chi'(v)\vert
 + \int m^2_{\e_2}(\cdot,dx,dv)\vert\chi'(v)\vert \\
&\quad +\int_x\big(S^1\chi(u^1)\sgn_{\frac 12}^+(u^1-u^2)_{\e_2}\big)_{\e_1} \\
&\quad - \int_x \big(S^2 \chi(u^2)\sgn_{\frac 12}^+(u^1-u^2)_{\e_1}\big)_{\e_2} .
\end{align*}
The same integrability properties also allow to let $\e_1,\e_2\to 0$ and to find
\begin{align*}
&\partial_t\int_{x,v} g^1(1-g^2)\chi(v)  \\
& \leq\int m^1(\cdot,dx,dv)\vert\chi'(v)\vert
 + \int m^2(\cdot,dx,dv)\vert\chi'(v)\vert \\
&\quad +\int_x S^1\chi(u^1)\sgn_{\frac 12}^+(u^1-u^2)
- \int_x S^2 \chi(u^2)\sgn_{\frac 12}^+(u^1-u^2) .
\end{align*}
We apply this inequality to a nonnegative test function $\zeta(t)$ and choose $\chi=\chi_n$ for a sequence $\chi_n\to 1$ a.e. with  $\chi_n'(v)\equiv 0$ for $\vert v \vert\leq n$ and $\vert \chi_n'\vert\leq 1$. Then the first two terms in the right-hand side are estimated by
\begin{equation*}
\Vert\zeta\Vert_{L^\infty}\cdot \sup_{\vert v\vert> n}\left(\mu^1(v) + \mu^2(v)\right),
\end{equation*}
which tends to $0$ as $n\to\infty$ since $\mu^j\in L^\infty_0$. The two last terms converge by dominated convergence, which yields
\begin{align*}
\partial_t\int_{x,v} g^1(1-g^2)  &\leq 
\int_x (S^1-S^2)\sgn_{\frac 12}^+(u^1-u^2).
\end{align*}
Applying this to a nonnegative test function $\zeta$ approaching $\zeta=\one_{[0,t]}$ and using that $u^j\in C([0,T],L^1(\R^n_x))$, we conclude that
\begin{equation*}
\int_x (u^1(t)-u^2(t))_+  \leq \int_x (u^1_0-u^2_0)_+  +\Vert S^1-S^2\Vert_{L^1((0,t)\times\R^n)}.
\end{equation*}

\textit{Existence}: Let $u_{0}^{\varepsilon}\in W^{1,\infty}(\R^{n})$, $S^{\varepsilon}\in W^{1,\infty}([0,T]\times\R^{n})$ with $u_{0}^{\varepsilon}\to u_{0}$ in $L^{1}(\R^{n})$ and $S^{\varepsilon}\to S$ in $L^{1}([0,T]\times\R^{n})$. Let $u^{\varepsilon}\in C([0,T];L^{1}(\R^{n}))$ be the corresponding unique entropy solution to \eqref{eq:intro_SCL} with initial condition $u_{0}^{\varepsilon}$ and force $S^{\varepsilon}$. By \eqref{eq:contraction} we have
\[
\|u^{\varepsilon}-u^{\delta}\|_{C([0,T];L^{1}(\R^{n}))}\le\|u_{0}^{\varepsilon}-u_{0}^{\delta}\|_{L^{1}(\R^{n})}+2\|S^{\varepsilon}-S^{\delta}\|_{L^{1}([0,T]\times\R^{n})}.
\]
Hence, there is a $u\in C([0,T];L^{1}(\R^{n}))$ such that $u^{\varepsilon}\to u$ in $C([0,T];L^{1}(\R^{n}))$ and thus almost everywhere for a subsequence. It is then easy to see that $u$ is a kinetic solution to \eqref{eq:intro_SCL}. \end{proof}
\begin{prop}
\label{prop:mass_bound}Let $u_{0}\in L^{1}(\R^{n})$, $S\in L^{1}([0,T]\times\R^{n})$ and $u$ be the corresponding entropy solution to \eqref{eq:intro_SCL}. For each $\alpha\in(0,1)$ and each $v_0 \in \R$ the measure $|v-v_0|^{\alpha-1}m$ has locally finite mass.\end{prop}
\begin{proof}
Let $\alpha\in(0,1)$. Let $\phi$ be a non-negative smooth compactly supported function in $(0,T)\times\R^{n}\times\R$. Then \eqref{eq:app_loc_kin} implies, with $\tilde{f}:=\phi f$, $\tilde{m}:=\phi m$,
\begin{equation}
\partial_{t}\tilde{f}+a(v)\cdot\nabla_{x}\tilde{f}=\partial_{v}\tilde{m}-(\partial_{v}\phi)m+\phi\delta_{u=v}S+\partial_{t}\phi f+(a(v)\cdot\nabla_{x}\phi)f.\label{eq:app_loc_kin_2}
\end{equation}
By translation we may assume $v_0 = 0$. We next choose a sequence of smooth, compactly supported functions $\eta^{\varepsilon}$ such that $\sgn(\eta^{\varepsilon}(v))=\sgn(v)$, $\eta^{\varepsilon}(v)\le\frac{1}{\alpha}\abs{v}^{\alpha}$ and $(\eta^{\varepsilon})^{'}\uparrow\abs{v}^{\alpha-1}$ pointwise. Multiplying \eqref{eq:app_loc_kin_2} by $\eta^{\varepsilon}$ and integrating yields 
\begin{align*}
\int_{x,v}(\eta^{\varepsilon}\tilde{f})(t) & +\int_{0}^{t}\int_{x,v}(\eta^{\varepsilon})'\tilde{m}=\int_{x,v}(\eta^{\varepsilon}\tilde{f})(0)\\
+\int_{0}^{t} & \int_{x,v}-\eta^{\varepsilon}(\partial_{v}\phi)m+\eta^{\varepsilon}\phi\delta_{u=v}S+\eta^{\varepsilon}\partial_{t}\phi f+\eta^{\varepsilon}(a(v)\cdot\nabla_{x}\phi)f.
\end{align*}

Since $m\ge0$, by Fatou's Lemma we may pass to the limit $\varepsilon\to0$ to obtain 
\begin{align*}
\int_{0}^{t}\int\abs{v}^{\alpha-1}\tilde{m}\,dxdvdr\le & C<\infty
\end{align*}
for some constant $C$ depending on $\|u_{0}\|_{L_{loc}^{1}},\|m\|_{\mathcal{M}_{loc}},\|S\|_{L_{loc}^{1}}$.
\end{proof}

\section{A basic estimate\label{sec:A-basic-estimate}}

From \cite{porous_medium} we recall the following basic $L^{p}$-estimate for certain Fourier multipliers. This result generalizes \cite[Lemma 2.2]{tadmor_tao_07} by taking into account possible $v$-regularity of $f$. This allows to avoid bootstrapping arguments in the application to scalar conservation laws. This is crucial in the case of scalar conservation laws with $L^{1}$-forcing, since in this case bootstrapping arguments do not apply.
\begin{lem}
\label{lem:Lpestim}Let $m(\tau',\xi',v):=i\tau'+ia(v)\cdot\xi'$, $\varphi,\phi$ be bounded, smooth functions, $\psi$ be a smooth cut-off function and $M_{\psi}$ be the Fourier multiplier with symbol $\varphi(\tau',\xi')\psi\left(\frac{m(\tau',\xi',v)}{\delta}\right)$. Then, for all $1<p\le2$, $\sigma\ge0$, $r\in[\frac{p'}{1+\sigma p'},p']\cap(1,\infty)$,
\begin{align*}
\|\int M_{\psi}f\phi\,dv\|_{L^{p}(\R_{t}\times\R^{n}_{x})} & \lesssim\|f\phi\|_{L^{p}(\R_{t}\times\R_{x}^{n};W^{\sigma,p}(\R_{v}))}\sup_{\tau',\xi'\in\mathrm{supp}\varphi}|\Omega_{m}(\tau',\xi',\delta)|^{\frac{1}{r}},
\end{align*}
where $\Omega_{m}(\tau',\xi',\delta)=\{v\in\supp\phi:\,|m(\tau',\xi',v)|\le\delta\}$. Moreover,
\begin{align*}
\|\int M_{\psi}f\phi\,dv\|_{\mathcal{M}(\R_{t}\times\R_{x}^{n})} & \lesssim\|f\phi\|_{\mathcal{M}(\R_{t}\times\R_{x}^{n})}.
\end{align*}

\end{lem}

Lemma~\ref{lem:Lpestim} relies on the fact that $\psi\left(\frac{i\tau'+ia(v)\cdot\xi'}{\delta}\right)$ is a bounded $L^{p}$ (and $\mathcal{M}$) multiplier uniformly in $v\in I$ and $\delta>0$ (the \textit{truncation property} in \cite{tadmor_tao_07}). This can be deduced, arguing as in \cite{diperna_lions_meyer_91}, from the invariance of the $L^{p}$ multiplier norm under partial dilations and the Marcinkiewicz multiplier theorem. For details we refer to \cite[Lemma~A.3]{porous_medium}. 

\bibliographystyle{abbrv}
\bibliography{averaging}

\end{document}